\documentclass[11pt]{article}

\usepackage[T1]{fontenc}
\usepackage{lmodern}
\usepackage{amsmath,amssymb,amsthm,mathtools}
\usepackage{enumitem}
\usepackage{microtype}
\usepackage{geometry}
\usepackage[hidelinks]{hyperref}
\newtheorem{theorem}{Theorem}[section]
\newtheorem{proposition}[theorem]{Proposition}
\newtheorem{lemma}[theorem]{Lemma}
\newtheorem{corollary}[theorem]{Corollary}
\theoremstyle{definition}
\newtheorem{definition}[theorem]{Definition}
\newtheorem{example}[theorem]{Example}
\theoremstyle{remark}
\newtheorem{remark}[theorem]{Remark}

\newcommand{\K}{\mathcal K}
\newcommand{\ImG}{\operatorname{Im}(g)}
\newcommand{\diam}{\operatorname{diam}}

\title{H\"older Selections under Stieltjes Clocks:\\Uniform Disconnectedness and Jump Dominance}
\author{%
\large Serkan \.{I}lter$^{1,*}$ \qquad
\large H\"ulya Duru$^{2}$ \qquad
\large Arkady Kitover$^{3}$ \qquad
\large Mehmet Orhon$^{4}$\\[7pt]
\normalsize $^{1,2}$Department of Mathematics, Faculty of Science, Istanbul University, Istanbul, T\"urkiye\\
\normalsize $^{3}$Department of Mathematics, Community College of Philadelphia, Philadelphia, Pennsylvania, USA\\
\normalsize $^{4}$Department of Mathematics and Statistics, University of New Hampshire, Durham, New Hampshire, USA\\
\\[5pt]
\small \texttt{ilters@istanbul.edu.tr} \qquad
\texttt{hduru@istanbul.edu.tr}\\
\small \texttt{akitover@ccp.edu} \qquad
\texttt{mo@unh.edu}\\[5pt]
\footnotesize $^{*}$Corresponding author: Serkan \.{I}lter, \texttt{ilters@istanbul.edu.tr}%
}
\date{}

\begin{document}
\maketitle

\begin{abstract}
It is known that, for ordinary H\"older multifunctions with $0<\alpha<1$, H\"older regularity alone does not guarantee the existence of a H\"older selection, and in general even a continuous selection may fail to exist. We ask how this situation changes when the parameter is measured through a Stieltjes clock. For $0<\alpha<1$, we characterize the clocks for which every compact-valued $g$-H\"older multifunction admits a $g$-H\"older selection. The determining condition is uniform disconnectedness of the clock image $\ImG$. For left-continuous and nondecreasing clocks, this condition is equivalent to uniform jump dominance: every positive clock increment contains a jump carrying a fixed positive fraction of that increment. Under this condition, a selection can be chosen through any prescribed point of the graph, with explicit control of its H\"older regularity. Conversely, when the condition fails, there exists a compact-valued $g$-H\"older multifunction with no $g$-H\"older selection. The results show that the selection property is governed not simply by the presence of jumps, but by how their sizes are distributed across scales.
\end{abstract}



\section{Introduction}\label{sec:intro}

A regular set-valued map need not admit a selection with the same regularity. This difficulty is especially pronounced for H\"older exponents below one, where even a compact-valued H\"older multifunction on an interval may fail to have a continuous selection. The situation can change when the parameter is measured through a Stieltjes clock that does not evolve continuously at every scale. The main question of this paper is therefore which Stieltjes clocks allow H\"older regularity to pass from a multifunction to one of its selections.

Classical selection theory shows that the existence and regularity of selections depend on the structure of both the domain and the values; convex-valued continuous-selection results are a basic example \cite{Michael1956}. Here the values are not assumed to be convex and may be arbitrary nonempty compact sets. Moreover, the aim is not merely to obtain a continuous selection, but to preserve the H\"older exponent of the multifunction.

For $0<\alpha<1$, there is a substantial obstruction. Chistyakov and Galkin constructed a compact-valued $\alpha$-H\"older multifunction on an interval with no continuous selection \cite{ChistyakovGalkin1998}. Later work on H\"older multifunctions and set-valued Young integration likewise shows that, without additional structure, H\"older regularity of the multifunction does not pass automatically to its selections; useful results may instead be obtained in weaker variation classes \cite{MichtaMotyl2020,MichtaMotyl2022}. Thus, in the regime $0<\alpha<1$, a natural question is which parameter sets remove this obstruction.

Stieltjes clocks provide a natural setting for this question. If a multifunction is $g$-H\"older, then points carrying the same value of $g$ must have the same set value. The selection problem can therefore be transferred from the original interval $I$ to the clock image $\ImG$. This leads to the central question:

\begin{quote}
For $0<\alpha<1$, which clock images $\ImG\subset\mathbb R$ guarantee that every compact-valued $\alpha$-H\"older multifunction admits a selection with the same H\"older exponent?
\end{quote}

The first main result gives a necessary and sufficient condition. If the clock image $\ImG$ is uniformly disconnected, then every compact-valued H\"older multifunction into an arbitrary metric space admits a selection with the same H\"older exponent. The selection can also be required to pass through any prescribed graph point, and its regularity is controlled by an explicit constant. Conversely, when the clock image is not uniformly disconnected, one can construct a compact-valued H\"older multifunction for which no selection has the required H\"older regularity.

This characterization shows that the selection problem is not governed only by the geometry of the values. What matters here is also how strongly the attainable clock values are separated from one another. Uniform disconnectedness provides exactly the separation needed across scales. This viewpoint is consistent with the close relation between uniformly disconnected spaces and ultrametric geometry \cite{DavidSemmes1997,TysonWu2005}, as well as with continuous-selection results on ultrametric domains without convexity assumptions \cite{StasyukTymchatyn2012}.

For left-continuous nondecreasing Stieltjes clocks, the geometric condition can be expressed directly in terms of the jumps of $g$. We introduce
\[
\kappa_g
=
\inf_{\substack{s<t\\g(s)<g(t)}}
\frac{\sup_{\tau\in[s,t)}\Delta^+g(\tau)}{g(t)-g(s)}.
\]
This quantity measures the fraction of a clock increment carried by its largest jump. One of the main results shows that
\[
\kappa_g>0
\]
is equivalent to uniform disconnectedness of $\ImG$. Thus the geometric condition has a simple interpretation: every positive clock increment must contain a jump carrying a fixed positive fraction of that increment.

Combining the two characterizations yields the main clock criterion. For $0<\alpha<1$, the condition $\kappa_g>0$ is equivalent to the universal $g$-H\"older selection property for compact-valued multifunctions over arbitrary metric target spaces. When $\kappa_g>0$, for every $\theta\in I$ and every $x_\theta\in F(\theta)$ there exists a $g$-H\"older selection $f$ with $f(\theta)=x_\theta$ and
\[
[f]_{\alpha,g}
\le
8\kappa_g^{-\alpha}[F]_{\alpha,g}.
\]
When $\kappa_g=0$, there exists a compact-valued $g$-H\"older multifunction with no $g$-H\"older selection. The same condition therefore determines both existence and quantitative control of the selection regularity.

The criterion distinguishes sharply between different jump structures. Finite-step clocks satisfy it, while infinite jump configurations may satisfy or fail it depending on how the jump sizes are distributed. The corresponding measure-theoretic interpretation is discussed later.

In summary, the paper links the H\"older selection problem for $0<\alpha<1$ to a directly identifiable property of the Stieltjes clock. Uniform disconnectedness of the clock image and uniform jump dominance are two forms of the same condition, and this condition determines whether H\"older regularity can be transferred from a compact-valued multifunction to a selection.

The paper is organized as follows. Section~\ref{sec:prelim} introduces the basic definitions and the reduction to the clock image. Section~\ref{sec:positive} proves the H\"older selection result under uniform disconnectedness, and Section~\ref{sec:necessity} establishes the converse and completes the characterization. Section~\ref{sec:clock} translates the geometric condition into the jump structure of a Stieltjes clock. Section~\ref{sec:examples} gives explicit examples, Section~\ref{sec:structural} studies the associated Stieltjes measure, and Section~\ref{sec:discussion} discusses the main consequences.

\section{Preliminaries}\label{sec:prelim}

Let $I=[a,b]$ with $a<b$, and let $(X,d)$ be a metric space. Denote by $\K(X)$ the family of all nonempty compact subsets of $X$. For $A,B\in\K(X)$, the Hausdorff distance is
\[
d_H(A,B)
=
\max\left\{
\sup_{x\in A}d(x,B),
\sup_{y\in B}d(y,A)
\right\}.
\]

For $\eta>0$ and $D\subset\mathbb R$, a subset $E\subset D$ is called an $\eta$-net of $D$ if, for every $x\in D$, there exists $y\in E$ such that $|x-y|\le\eta$.

We denote by $c_0(\mathbb R^2)$ the Banach space of all sequences $x=(x_n)_{n\ge1}$ with $x_n\in\mathbb R^2$ and $\|x_n\|\to0$, endowed with the supremum norm $\|x\|_\infty=\sup_{n\ge1}\|x_n\|$.

When $g:I\to\mathbb R$ is left-continuous and nondecreasing, for every $\tau\in[a,b)$ set
\[
\Delta^+g(\tau)=g(\tau+)-g(\tau),
\qquad
D_g=\{\tau\in[a,b):\Delta^+g(\tau)>0\}.
\]
Let $\mu_g$ denote the associated Lebesgue--Stieltjes measure, normalized, for every $s,t\in I$ with $s<t$, by
\[
\mu_g([s,t))=g(t)-g(s).
\]
Then $\mu_g(\{\tau\})=\Delta^+g(\tau)$ for every $\tau\in D_g$. We write
\[
\mu_g=\mu_g^a+\mu_g^c,
\qquad
\mu_g^a=\sum_{\tau\in D_g}\Delta^+g(\tau)\,\delta_\tau,
\]
where $\delta_\tau$ is the Dirac point mass at $\tau$ and $\mu_g^c$ is nonatomic.

Let $g:I\to\mathbb R$ have bounded image and let $\alpha>0$. We denote the image of $g$ by $\ImG=g(I)$. A map $F:I\to\K(X)$ is $g$-H\"older of exponent $\alpha$ if there exists $L\ge0$ such that, for every $s,t\in I$,
\[
d_H(F(s),F(t))
\le
L|g(t)-g(s)|^\alpha.
\]
The least such constant is denoted by $[F]_{\alpha,g}$. For a single-valued map $f:I\to X$, the notation $[f]_{\alpha,g}$ is defined analogously, with $d$ in place of $d_H$.

A selection of $F$ is a map $f:I\to X$ satisfying $f(t)\in F(t)$ for every $t\in I$. If $\theta\in I$ and $x_\theta\in F(\theta)$ are fixed, a selection satisfying $f(\theta)=x_\theta$ is called a prescribed-point selection.

\begin{lemma}[Clock-image factorization]\label{lem:factorization}
Let $\alpha>0$, let $g:I\to\mathbb R$, and suppose that $F:I\to\K(X)$ is $g$-H\"older of exponent $\alpha$. Then $F$ is constant on every level set of $g$. Consequently, there exists a unique map
\[
\widehat F:\ImG\to\K(X)
\]
such that $F=\widehat F\circ g$. Moreover,
\[
[\widehat F]_{\alpha,\ImG}=[F]_{\alpha,g},
\]
where
\[
[\widehat F]_{\alpha,\ImG}
=
\sup_{\substack{u,v\in\ImG\\u\ne v}}
\frac{d_H(\widehat F(u),\widehat F(v))}{|u-v|^\alpha}.
\]
\end{lemma}

\begin{proof}
If $g(s)=g(t)$, then
\[
d_H(F(s),F(t))
\le
[F]_{\alpha,g}|g(t)-g(s)|^\alpha=0,
\]
so $F(s)=F(t)$. Therefore $\widehat F(g(t)):=F(t)$ is well defined and $F=\widehat F\circ g$.

For $u,v\in\ImG$, choose $s,t\in I$ with $u=g(s)$ and $v=g(t)$. Then
\[
d_H(\widehat F(u),\widehat F(v))
=
d_H(F(s),F(t))
\le
[F]_{\alpha,g}|u-v|^\alpha,
\]
which gives $[\widehat F]_{\alpha,\ImG}\le [F]_{\alpha,g}$. The reverse inequality follows by applying the H\"older estimate for $\widehat F$ to $g(s)$ and $g(t)$. Hence the seminorms are equal.
\end{proof}

\begin{remark}[Selections under the clock-image reduction]\label{rem:selection-factorization}
The same factorization applies to single-valued $g$-H\"older maps. In particular, if $\widehat f:\ImG\to X$ satisfies $\widehat f(u)\in\widehat F(u)$ for all $u\in\ImG$, then $f=\widehat f\circ g$ is a selection of $F$ and
\[
[f]_{\alpha,g}=[\widehat f]_{\alpha,\ImG}.
\]
If $u_\theta=g(\theta)$ and $\widehat f(u_\theta)=x_\theta$, then $f(\theta)=x_\theta$.
\end{remark}

\begin{definition}[Uniform disconnectedness]\label{def:uniform-disconnected}
Following the standard chain formulation of uniform disconnectedness \cite{DavidSemmes1997,TysonWu2005}, let $D\subset\mathbb R$ and $c\in(0,1]$. The set $D$ is called $c$-uniformly disconnected if, for every pair of distinct points $u,v\in D$ and every finite chain $u=u_0,u_1,\ldots,u_N=v$ with $u_i\in D$, one has
\[
\max_{1\le i\le N}|u_i-u_{i-1}|
\ge
c|u-v|.
\]
The set $D$ is uniformly disconnected if it is $c$-uniformly disconnected for some $c>0$.
\end{definition}

We introduce the \emph{chain bottleneck} to quantify the smallest possible value of the largest step along a finite chain joining two points. Every such chain must contain at least one step of this size.

\begin{definition}[Chain bottleneck]\label{def:bottleneck}
Let $D\subset\mathbb R$. For every $u,v\in D$, define
\[
\beta_D(u,v)
=
\inf\left\{
\max_{1\le i\le N}|u_i-u_{i-1}|:
 u=u_0,\ldots,u_N=v,\ u_i\in D
\right\},
\]
with $\beta_D(u,u)=0$. Thus $D$ is $c$-uniformly disconnected precisely when, for every $u,v\in D$,
\[
\beta_D(u,v)\ge c|u-v|.
\]
\end{definition}

The next elementary proposition records the standard chain-to-ultrametric mechanism associated with uniform disconnectedness; compare \cite{DavidSemmes1997,TysonWu2005}.

\begin{proposition}[The bottleneck ultrametric]\label{prop:bottleneck-ultrametric}
Let $D\subset\mathbb R$ be $c$-uniformly disconnected for some $c\in(0,1]$. Then $\beta_D$ is an ultrametric on $D$, and for every $u,v\in D$,
\[
c|u-v|\le\beta_D(u,v)\le|u-v|.
\]
Moreover, for every $\alpha>0$, the function $\delta_\alpha(u,v)=\beta_D(u,v)^\alpha$ is an ultrametric on $D$ and for every $u,v\in D$,
\[
c^\alpha|u-v|^\alpha
\le
\delta_\alpha(u,v)
\le
|u-v|^\alpha.
\]
\end{proposition}

\begin{proof}
The upper bound follows from the one-step chain $u,v$, and the lower bound is exactly the $c$-uniform disconnectedness condition. Symmetry is immediate.

Let $u,v,w\in D$ and $\varepsilon>0$. Choose a chain from $u$ to $v$ whose largest step is smaller than $\beta_D(u,v)+\varepsilon$ and a chain from $v$ to $w$ whose largest step is smaller than $\beta_D(v,w)+\varepsilon$. Concatenating the chains gives
\[
\beta_D(u,w)
\le
\max\{\beta_D(u,v)+\varepsilon,\beta_D(v,w)+\varepsilon\}.
\]
Letting $\varepsilon\downarrow0$ yields the strong triangle inequality. Finally, the lower bound implies $\beta_D(u,v)>0$ whenever $u\ne v$, so $\beta_D$ is an ultrametric. Since a positive power preserves the strong triangle inequality, $\delta_\alpha=\beta_D^\alpha$ is also an ultrametric, and the stated bounds follow by raising the preceding two-sided estimate to the power $\alpha$.
\end{proof}

Without uniform disconnectedness, $\beta_D$ need not separate distinct points; for example, two points of an interval can be joined by chains whose maximal step goes to zero.

\section{Uniformly disconnected domains and quantitative prescribed-point selections}\label{sec:positive}

The ultrametric structure in Proposition~\ref{prop:bottleneck-ultrametric} yields controlled selections. We begin with a finite prescribed-point selection lemma.

\begin{lemma}[Finite prescribed-point selection on an ultrametric set]\label{lem:finite-ultrametric-selection}
Let $(E,\delta)$ be a nonempty finite ultrametric space, let $(X,d)$ be a metric space, and let $F:E\to\K(X)$. Suppose that, for some $K\ge0$, the following estimate holds for every $u,v\in E$:
\[
d_H(F(u),F(v))\le K\delta(u,v).
\]
Fix $u_*\in E$ and $x_*\in F(u_*)$. Then there exists a selection $f:E\to X$ such that $f(u_*)=x_*$ and for every $u,v\in E$,
\[
d(f(u),f(v))\le8K\delta(u,v).
\]
\end{lemma}

\begin{proof}
If $E=\{u_*\}$, the assertion is immediate. Assume $|E|\ge2$ and set
\[
R=\diam_\delta(E)>0,
\qquad r_k=2^{-k}R\quad(k\ge0).
\]
Choose $k_0$ so large that
\[
r_{k_0}<\min_{u\ne v}\delta(u,v).
\]
For $k=0,1,\ldots,k_0$, declare $u,v\in E$ equivalent when $\delta(u,v)\le r_k$. Since $\delta$ is an ultrametric, these are equivalence relations. Let $\mathcal C_k$ be the family of the corresponding equivalence classes, which we call clusters. The families are nested, and every cluster in $\mathcal C_{k_0}$ is a singleton.

For each cluster $C$, choose a representative $r_C\in C$, with $r_C=u_*$ whenever $u_*\in C$. We assign a point $x_C\in F(r_C)$ recursively through these nested families. At level $0$, the unique cluster is $E$, and we set $x_E=x_*$. Suppose $P\in\mathcal C_k$ has already been assigned $x_P\in F(r_P)$ and let $C\in\mathcal C_{k+1}$ satisfy $C\subset P$. If $u_*\in C$, set $x_C=x_*$. Otherwise, compactness of $F(r_C)$ allows us to choose $x_C\in F(r_C)$ with
\[
d(x_C,x_P)=d(x_P,F(r_C)).
\]
Since $r_C,r_P\in P$,
\[
\delta(r_C,r_P)\le r_k,
\]
and hence
\[
d(x_C,x_P)
\le
d_H(F(r_C),F(r_P))
\le
Kr_k.
\]
For $u\in E$, the level-$k_0$ cluster containing $u$ is $\{u\}$; define
\[
f(u)=x_{\{u\}}.
\]
Then $f(u)\in F(u)$ and $f(u_*)=x_*$.

Let $u\ne v$, and let $m$ be the largest level at which $u$ and $v$ belong to the same cluster. They lie in distinct clusters at level $m+1$, so
\[
\delta(u,v)>r_{m+1}=\frac{r_m}{2},
\qquad r_m<2\delta(u,v).
\]
Following the nested clusters containing $u$ from their common level-$m$ cluster down to the singleton $\{u\}$, the total displacement is bounded by
\[
Kr_m+Kr_{m+1}+\cdots\le2Kr_m,
\]
and the same bound holds for the nested clusters leading to $\{v\}$. Therefore
\[
d(f(u),f(v))\le4Kr_m<8K\delta(u,v).
\]
\end{proof}

\begin{remark}\label{rem:constant-eight}
The factor $8$ is the explicit bound produced by the dyadic tree construction; no sharpness claim is made for the selection problem itself. If consecutive scales have ratio $q\in(0,1)$, the same calculation gives $2/[q(1-q)]$, which is minimized at $q=1/2$ within this geometric-scale construction.
\end{remark}

\begin{theorem}[Quantitative prescribed-point selection]\label{thm:positive-selection}
Let $D\subset\mathbb R$ be $c$-uniformly disconnected, let $(X,d)$ be a metric space, and let $\alpha>0$. Suppose that $F:D\to\K(X)$ satisfies, for every $u,v\in D$,
\[
d_H(F(u),F(v))\le L|u-v|^\alpha.
\]
Then, for every $u_*\in D$ and $x_*\in F(u_*)$, there exists a selection $f:D\to X$ satisfying $f(u_*)=x_*$ and for every $u,v\in D$,
\[
d(f(u),f(v))
\le
8c^{-\alpha}L|u-v|^\alpha.
\]
Consequently,
\[
[f]_{\alpha,D}
\le
8c^{-\alpha}[F]_{\alpha,D}.
\]
\end{theorem}

\begin{proof}
For every $u,v\in D$, Proposition~\ref{prop:bottleneck-ultrametric} gives
\[
c|u-v|\le\beta_D(u,v)\le|u-v|.
\]
Define $\delta:D\times D\to[0,\infty)$ by $\delta(u,v)=\beta_D(u,v)^\alpha$. Then $\delta$ is an ultrametric and for every $u,v\in D$,
\[
|u-v|^\alpha\le c^{-\alpha}\delta(u,v).
\]
Thus, for every $u,v\in D$,
\[
d_H(F(u),F(v))
\le
Lc^{-\alpha}\delta(u,v).
\]
For every finite $E\subset D$ containing $u_*$, Lemma~\ref{lem:finite-ultrametric-selection} gives a prescribed-point selection $f_E:E\to X$ satisfying, for every $u,v\in E$,
\[
d(f_E(u),f_E(v))
\le
8Lc^{-\alpha}\delta(u,v)
\le
8Lc^{-\alpha}|u-v|^\alpha.
\]

To extend this construction to the whole domain, consider the product
\[
\mathcal P=\prod_{u\in D}F(u)
\]
with the product topology; its elements are maps $f:D\to X$ satisfying $f(u)\in F(u)$ for every $u\in D$. It is compact by Tychonoff's theorem. For distinct $u,v\in D$, let
\[
C_{u,v}
=
\{f\in\mathcal P:
 d(f(u),f(v))\le8c^{-\alpha}L|u-v|^\alpha\},
\]
and let
\[
C_*=\{f\in\mathcal P:f(u_*)=x_*\}.
\]
These sets are closed because the coordinate projections are continuous in the product topology and the distance function is continuous. Every finite subfamily of these constraints involves only finitely many coordinates; after adjoining $u_*$ if necessary, Lemma~\ref{lem:finite-ultrametric-selection} supplies the required values on those coordinates, while arbitrary values may be selected in the remaining nonempty fibers. Hence the family has the finite intersection property. Compactness gives
\[
C_*\cap\bigcap_{u\ne v}C_{u,v}\ne\varnothing,
\]
and any element of this intersection is the desired global selection.
\end{proof}

\section{Necessity and the H\"older-selection characterization}\label{sec:necessity}

For $0<\alpha<1$, let $H:[0,1]\to\K(\mathbb R^2)$ denote, throughout this section, the Chistyakov--Galkin counterexample from Proposition~8.2 of \cite{ChistyakovGalkin1998}, after an affine change of the parameter interval. Thus, $H$ is an $\alpha$-H\"older multifunction that admits no continuous selection. Write $L_H=[H]_\alpha<\infty$. Its values are contained in the unit circle, so
\[
K_H:=\bigcup_{t\in[0,1]}H(t)
\]
is bounded. Set
\[
R_H:=\sup_{y\in K_H}\|y\|.
\]

\begin{lemma}[Finite-net obstruction]\label{lem:finite-net-obstruction}
For every $M>0$ there exists $\eta_M>0$ such that, whenever $E\subset[0,1]$ is a finite $\eta_M$-net of $[0,1]$ containing $0$ and $1$, no selection $h:E\to\mathbb R^2$ of $H|_E$ can satisfy
\[
[h]_{\alpha,E}\le M.
\]
\end{lemma}

\begin{proof}
Suppose the conclusion fails for some $M>0$. Then for every $n$ there is a finite $1/n$-net $E_n\subset[0,1]$, containing $0$ and $1$, and a selection $h_n:E_n\to\mathbb R^2$ with $[h_n]_{\alpha,E_n}\le M$.

Let $Q\subset[0,1]$ be countable and dense. For each $q\in Q$, choose $q_n\in E_n$ with $|q_n-q|\le1/n$. Since $h_n(q_n)\in K_H$ and $K_H$ is bounded, a diagonal subsequence defines a map $h:Q\to\mathbb R^2$ such that, for every $q\in Q$,
\[
h_n(q_n)\to h(q).
\] For $q,r\in Q$,
\[
\|h_n(q_n)-h_n(r_n)\|
\le
M|q_n-r_n|^\alpha,
\]
so
\[
\|h(q)-h(r)\|\le M|q-r|^\alpha.
\]
Moreover,
\[
d(h_n(q_n),H(q))
\le
d_H(H(q_n),H(q))
\le
L_H|q_n-q|^\alpha\to0,
\]
and therefore $h(q)\in H(q)$ because $H(q)$ is closed.

Since $\mathbb R^2$ is complete, the H\"older estimate extends $h$ uniquely from $Q$ to a continuous, indeed $\alpha$-H\"older, map $h:[0,1]\to\mathbb R^2$. If $q_k\to t$, then $h(q_k)\to h(t)$ and $d_H(H(q_k),H(t))\to0$, so closedness of $H(t)$ yields $h(t)\in H(t)$. This gives a continuous selection of $H$, a contradiction.
\end{proof}

\begin{lemma}[Fine normalized chains]\label{lem:fine-chains}
Let $D\subset\mathbb R$ be bounded and not uniformly disconnected. Then, for every $\eta>0$, there exist $a,b\in D$ with $a<b$ and a finite set $E\subset D\cap[a,b]$ containing $a$ and $b$ such that
\[
\widetilde E
=
\left\{\frac{u-a}{b-a}:u\in E\right\}
\]
is an $\eta$-net of $[0,1]$.
\end{lemma}

\begin{proof}
Since $D$ is not uniformly disconnected, for every $\varepsilon>0$ there exist distinct $u,v\in D$ and a finite chain $u=z_0,z_1,\ldots,z_N=v$ with $z_i\in D$ for every $i=0,\ldots,N$ such that
\[
\max_i|z_i-z_{i-1}|<\varepsilon|u-v|.
\]
Let
\[
a=\min_i z_i,
\qquad b=\max_i z_i.
\]
Then $b-a\ge|u-v|$. Arrange the distinct chain points increasingly,
\[
a=w_0<w_1<\cdots<w_m=b.
\]
Set $E=\{w_0,\ldots,w_m\}\subset D\cap[a,b]$. For each $j$, the original chain must cross the gap $(w_{j-1},w_j)$ in one adjacent step, and therefore
\[
w_j-w_{j-1}
\le
\max_i|z_i-z_{i-1}|
<
\varepsilon(b-a).
\]
Choosing $\varepsilon<\eta$ shows that the normalized set has consecutive gaps smaller than $\eta$ and contains $0,1$; hence it is an $\eta$-net.
\end{proof}

\begin{theorem}[A H\"older-selection obstruction]\label{thm:single-obstruction}
Let $0<\alpha<1$, and let $D\subset\mathbb R$ be nonempty, bounded, and not uniformly disconnected. Then there exists an $\alpha$-H\"older map
\[
\mathcal F:D\to\K\bigl(c_0(\mathbb R^2)\bigr)
\]
that admits no $\alpha$-H\"older selection.
\end{theorem}

\begin{proof}
For $n\ge1$, set
\[
M_n:=n2^n,
\qquad \eta_n:=\eta_{M_n},
\]
where $\eta_{M_n}$ is supplied by Lemma~\ref{lem:finite-net-obstruction}. By Lemma~\ref{lem:fine-chains}, choose $a_n,b_n\in D$ with $a_n<b_n$, put $r_n=b_n-a_n$, and choose a finite $E_n^*\subset D\cap[a_n,b_n]$ such that
\[
E_n
=
\left\{\frac{u-a_n}{r_n}:u\in E_n^*\right\}
\]
is an $\eta_n$-net of $[0,1]$.

To transfer the obstruction on $[0,1]$ to the interval $[a_n,b_n]\subset D$, define $\phi_n:D\to[0,1]$ by
\[
\phi_n(u)=
\begin{cases}
0,&u\le a_n,\\[1mm]
\dfrac{u-a_n}{r_n},&a_n<u<b_n,\\[2mm]
1,&u\ge b_n,
\end{cases}
\]
so that, for every $u,v\in D$,
\[
|\phi_n(u)-\phi_n(v)|\le\frac{|u-v|}{r_n}.
\]
Set
\[
\lambda_n=2^{-n}r_n^\alpha.
\]
The factor $2^{-n}$ places the resulting sequences in $c_0(\mathbb R^2)$, while $r_n^\alpha$ compensates for the rescaling introduced by $\phi_n$. For $u\in D$, define
\[
\mathcal F(u)
=
\left\{
(\lambda_n y_n)_{n\ge1}:
 y_n\in H(\phi_n(u))\text{ for every }n
\right\}
\subset c_0(\mathbb R^2).
\]

We first verify compactness of each value. Since $R_H<\infty$ and $D$ is bounded,
\[
\lambda_nR_H
\le
2^{-n}(\diam D)^\alpha R_H\to0.
\]
Thus all tails go to zero uniformly over $\mathcal F(u)$. Every finite coordinate block ranges over a finite product of compact sets. A diagonal subsequence argument, together with the preceding bound, therefore yields convergence in the supremum norm to a point of $\mathcal F(u)$. Hence $\mathcal F(u)$ is compact in $c_0(\mathbb R^2)$.

For $u,v\in D$ and every $n$,
\[
d_H\bigl(\lambda_nH(\phi_n(u)),\lambda_nH(\phi_n(v))\bigr)
\le
\lambda_nL_H|\phi_n(u)-\phi_n(v)|^\alpha
\le
2^{-n}L_H|u-v|^\alpha.
\]
Coordinatewise nearest-point choices remain in $c_0(\mathbb R^2)$, since their $n$th coordinates are bounded by $\lambda_nR_H\to0$, and imply
\[
d_H(\mathcal F(u),\mathcal F(v))
\le
\sup_{n\ge1}2^{-n}L_H|u-v|^\alpha
\le
\frac{L_H}{2}|u-v|^\alpha.
\]
Thus $\mathcal F$ is $\alpha$-H\"older.

Suppose that $f:D\to c_0(\mathbb R^2)$ is an $\alpha$-H\"older selection and write
\[
M=[f]_{\alpha,D}<\infty.
\]
Let $\pi_n:c_0(\mathbb R^2)\to\mathbb R^2$ be the $n$th coordinate projection and define $h_n:D\to\mathbb R^2$ by
\[
h_n(u)=\lambda_n^{-1}\pi_nf(u).
\]
Since $f(u)\in\mathcal F(u)$, one has $h_n(u)\in H(\phi_n(u))$ for every $u\in D$; hence $h_n:D\to\mathbb R^2$ is a selection of $H\circ\phi_n$. Define $\widetilde h_n:E_n\to\mathbb R^2$ by
\[
\widetilde h_n(x)=h_n(a_n+r_nx).
\]
This is a selection of $H|_{E_n}$, and for every $x,y\in E_n$,
\[
\begin{aligned}
\|\widetilde h_n(x)-\widetilde h_n(y)\|
&\le
\lambda_n^{-1}Mr_n^\alpha|x-y|^\alpha\\
&=
M2^n|x-y|^\alpha.
\end{aligned}
\]
Hence $[\widetilde h_n]_{\alpha,E_n}\le M2^n$. Choosing an integer $n>M$ gives
\[
M2^n<n2^n=M_n,
\]
contradicting Lemma~\ref{lem:finite-net-obstruction}. Therefore $\mathcal F$ admits no $\alpha$-H\"older selection.
\end{proof}

We remark that Theorem~\ref{thm:positive-selection} holds for every $\alpha>0$, whereas the converse implication \textup{(ii)}$\Rightarrow$\textup{(i)} in the following theorem requires $0<\alpha<1$.

\begin{theorem}[Characterization of universal H\"older-selection domains]\label{thm:domain-characterization}
Let $0<\alpha<1$, and let $D\subset\mathbb R$ be nonempty and bounded. The following are equivalent.
\begin{enumerate}[label=\textup{(\roman*)}]
\item $D$ is uniformly disconnected.
\item For every metric space $(X,d)$ and every compact-valued $\alpha$-H\"older map $F:D\to\K(X)$, there exists an $\alpha$-H\"older selection of $F$.
\item There exists $C=C(D,\alpha)<\infty$ such that, for every metric space $(X,d)$, every compact-valued $\alpha$-H\"older map $F:D\to\K(X)$, every $u_*\in D$, and every $x_*\in F(u_*)$, there exists a selection $f:D\to X$ satisfying
\[
f(u_*)=x_*,
\qquad
[f]_{\alpha,D}\le C[F]_{\alpha,D}.
\]
\end{enumerate}
If $D$ is $c$-uniformly disconnected, one may take $C=8c^{-\alpha}$.
\end{theorem}

\begin{proof}
The implication (iii)$\Rightarrow$(ii) is immediate. Theorem~\ref{thm:positive-selection} gives (i)$\Rightarrow$(iii), with $C=8c^{-\alpha}$ when $D$ is $c$-uniformly disconnected. If (i) fails, Theorem~\ref{thm:single-obstruction} gives an $\alpha$-H\"older compact-valued map into $c_0(\mathbb R^2)$ with no $\alpha$-H\"older selection, so (ii) fails. Hence (ii)$\Rightarrow$(i).
\end{proof}

\begin{remark}[Scope of the converse]\label{rem:fixed-target}
The necessity part of Theorem~\ref{thm:domain-characterization} concerns a universal statement over target spaces. The obstruction in Theorem~\ref{thm:single-obstruction} takes values in the infinite-dimensional Banach space $c_0(\mathbb R^2)$. Thus the theorem does not by itself characterize domains for which every compact-valued map into a fixed finite-dimensional target such as $\mathbb R^d$ admits a H\"older selection.
\end{remark}

\section{Characterization of admissible Stieltjes clocks}\label{sec:clock}

Throughout this section, let $I=[a,b]$ and let $g:I\to\mathbb R$ be left-continuous and nondecreasing. With the jump notation introduced in Section~\ref{sec:prelim}, define $J_g:\{(s,t)\in I^2:s<t\}\to[0,\infty)$ by
\[
J_g(s,t)=\sup_{\tau\in[s,t)}\Delta^+g(\tau).
\]
If $g$ is nonconstant, define
\[
\kappa_g
=
\inf_{\substack{a\le s<t\le b\\g(s)<g(t)}}
\frac{J_g(s,t)}{g(t)-g(s)}.
\]

\begin{lemma}[Chain bottleneck and gaps on the real line]\label{lem:bottleneck-gap}
Let $D\subset\mathbb R$ and let $u,v\in D$ with $u<v$. Define $\Gamma_D(u,v)\in[0,\infty)$ by
\[
\Gamma_D(u,v)
=
\sup\left(
\left\{
q-p:
 u\le p<q\le v,
\ D\cap(p,q)=\varnothing
\right\}\cup\{0\}
\right).
\]
Then
\[
\beta_D(u,v)=\Gamma_D(u,v).
\]
\end{lemma}

\begin{proof}
If $(p,q)\subset[u,v]$ is disjoint from $D$, every finite chain in $D$ joining $u$ to $v$ must cross this gap in one step. Hence its maximal step is at least $q-p$, and therefore $\beta_D(u,v)\ge\Gamma_D(u,v)$.

For the reverse inequality, fix $r>\Gamma_D(u,v)$ and let $C\subset D\cap[u,v]$ be the set of points reachable from $u$ by a finite chain all of whose steps are strictly smaller than $r$. Suppose $v\notin C$ and put $\xi=\sup C$. Then $v-\xi\ge r$; otherwise a point of $C$ sufficiently close to $\xi$ could be joined to $v$ by one further step smaller than $r$.

We claim that $D\cap(\xi,\xi+r)=\varnothing$. If $y\in D\cap(\xi,\xi+r)$, choose $x\in C$ sufficiently close to $\xi$ that $y-x<r$. Appending $y$ to a chain from $u$ to $x$ would imply $y\in C$, contradicting $y>\xi$. Hence $(\xi,\xi+r)$ is a gap contained in $[u,v]$, so $\Gamma_D(u,v)\ge r$, contrary to the choice of $r$. Thus $v\in C$, which gives $\beta_D(u,v)\le r$. Since the argument applies to every $r>\Gamma_D(u,v)$, it follows that $\beta_D(u,v)\le\Gamma_D(u,v)$. Together with the reverse inequality proved above, this gives the equality.
\end{proof}

\begin{proposition}[Exact bottleneck--jump identity]\label{prop:bottleneck-jump}
Let $a\le s<t\le b$ and suppose $g(s)<g(t)$. Then
\[
\beta_{\ImG}(g(s),g(t))
=
J_g(s,t).
\]
\end{proposition}

\begin{proof}
Let $\tau\in[s,t)$. By monotonicity, for $r\in I$,
\[
g(r)\le g(\tau)\quad(r\le\tau),
\qquad
g(r)\ge g(\tau+)\quad(r>\tau).
\]
Hence
\[
\ImG\cap(g(\tau),g(\tau+))=\varnothing,
\]
so Lemma~\ref{lem:bottleneck-gap} gives
\[
\Gamma_{\ImG}(g(s),g(t))\ge\Delta^+g(\tau).
\]
Taking the supremum yields
\[
\Gamma_{\ImG}(g(s),g(t))\ge J_g(s,t).
\]

Conversely, let $g(s)\le p<q\le g(t)$ and suppose $\ImG\cap(p,q)=\varnothing$. Set
\[
\tau=\sup\{r\in[s,t]:g(r)\le p\}.
\]
The set is nonempty. Moreover $\tau<t$; otherwise left-continuity at $t$ would imply $g(t)\le p$, contradicting $g(t)\ge q>p$. If $\tau=s$, then $g(\tau)=g(s)\le p$; if $\tau>s$, left-continuity and an approximating sequence from the defining set again give $g(\tau)\le p$.

For every $r\in(\tau,t]$, one has $g(r)>p$. Since no value of $g$ lies in $(p,q)$, it follows that $g(r)\ge q$ for every $r\in(\tau,t]$. Taking the right limit gives $g(\tau+)\ge q$, and hence
\[
\Delta^+g(\tau)
=g(\tau+)-g(\tau)
\ge q-p.
\]
Thus every gap has length at most $J_g(s,t)$, so
\[
\Gamma_{\ImG}(g(s),g(t))\le J_g(s,t).
\]
Lemma~\ref{lem:bottleneck-gap} completes the proof.
\end{proof}

\begin{corollary}[Uniform disconnectedness in terms of clock jumps]\label{cor:kappa-ud}
Let $g$ be nonconstant. Then $\ImG$ is uniformly disconnected if and only if $\kappa_g>0$. More precisely, the optimal uniform-disconnectedness constant of $\ImG$ equals $\kappa_g$.
\end{corollary}

\begin{proof}
By Definition~\ref{def:uniform-disconnected}, $\ImG$ is $c$-uniformly disconnected exactly when
\[
\beta_{\ImG}(u,v)\ge c|u-v|
\]
for all distinct $u,v\in\ImG$. Every ordered pair $u<v$ in $\ImG$ can be written as $u=g(s)$ and $v=g(t)$ with $s<t$. Proposition~\ref{prop:bottleneck-jump} therefore turns this condition into
\[
J_g(s,t)\ge c[g(t)-g(s)]
\]
for every $s<t$ with $g(s)<g(t)$. Taking the infimum of these ratios shows that every admissible uniform-disconnectedness constant $c$ satisfies $c\le \kappa_g$. Conversely, the definition of $\kappa_g$ gives $J_g(s,t)\ge \kappa_g[g(t)-g(s)]$ for every $s<t$ with $g(s)<g(t)$. Hence, if $\kappa_g>0$, Proposition~\ref{prop:bottleneck-jump} shows that $\ImG$ is $\kappa_g$-uniformly disconnected. Therefore $\kappa_g$ is the optimal constant.
\end{proof}

Combining Theorem~\ref{thm:domain-characterization} with Corollary~\ref{cor:kappa-ud} expresses the universal H\"older-selection property directly in terms of the jump structure of the Stieltjes clock $g$.

\begin{theorem}[Characterization of Stieltjes clocks preserving H\"older selections]\label{thm:clock-characterization}
Let $0<\alpha<1$, and let $g:I\to\mathbb R$ be nonconstant, left-continuous, and nondecreasing. The following statements are equivalent.
\begin{enumerate}[label=\textup{(\roman*)}]
\item $\kappa_g>0$.
\item For every metric space $(X,d)$ and every compact-valued $g$-H\"older map $F:I\to\K(X)$ of exponent $\alpha$, there exists a $g$-H\"older selection of $F$.
\item There exists $C=C(g,\alpha)<\infty$ such that, for every metric space $(X,d)$, every compact-valued $g$-H\"older map $F:I\to\K(X)$ of exponent $\alpha$, every $\theta\in I$, and every $x_\theta\in F(\theta)$, there exists a selection $f:I\to X$ satisfying
\[
f(\theta)=x_\theta,
\qquad
[f]_{\alpha,g}\le C[F]_{\alpha,g}.
\]
\end{enumerate}
If $\kappa_g>0$, one may take
\[
C=8\kappa_g^{-\alpha}.
\]
\end{theorem}

\begin{proof}
By Lemma~\ref{lem:factorization}, every $g$-H\"older map $F:I\to\K(X)$ factors uniquely as $F=\widehat F\circ g$ with $\widehat F:\ImG\to\K(X)$, and
\[
[\widehat F]_{\alpha,\ImG}=[F]_{\alpha,g}.
\]
If $\kappa_g>0$, Corollary~\ref{cor:kappa-ud} shows that $\ImG$ is $\kappa_g$-uniformly disconnected. Theorem~\ref{thm:positive-selection}, applied to $\widehat F$, yields for every prescribed $x_\theta\in\widehat F(g(\theta))$ a selection $\widehat f:\ImG\to X$ satisfying
\[
[\widehat f]_{\alpha,\ImG}
\le
8\kappa_g^{-\alpha}[\widehat F]_{\alpha,\ImG}.
\]
Setting $f=\widehat f\circ g$ and using Remark~\ref{rem:selection-factorization} proves (i)$\Rightarrow$(iii). The implication (iii)$\Rightarrow$(ii) is immediate.

Suppose (i) fails. Then $\kappa_g=0$, so Corollary~\ref{cor:kappa-ud} shows that $\ImG$ is not uniformly disconnected. Theorem~\ref{thm:single-obstruction} provides a compact-valued $\alpha$-H\"older map
\[
\widehat F:\ImG\to\K(X)
\]
for some metric space $X$, with no $\alpha$-H\"older selection. Define $F=\widehat F\circ g$. Then $F$ is $g$-H\"older of exponent $\alpha$. If $f:I\to X$ were a $g$-H\"older selection of $F$, then, for every $s,t\in I$ with $g(s)=g(t)$, one would have $f(s)=f(t)$; hence $f$ would factor through $g$ as $f=\widehat f\circ g$ for a map $\widehat f:\ImG\to X$. Remark~\ref{rem:selection-factorization} would then make $\widehat f$ an $\alpha$-H\"older selection of $\widehat F$, a contradiction. Hence (ii) fails and (ii)$\Rightarrow$(i).
\end{proof}

\begin{remark}[Constant clocks]\label{rem:constant-clocks}
If $g$ is constant, every $g$-H\"older compact-valued map is itself constant, and every prescribed graph point extends to a constant selection. The parameter $\kappa_g$ is therefore needed only for nonconstant clocks.
\end{remark}

\begin{remark}[Continuous clocks]\label{rem:continuous-clocks}
If $g$ is continuous and nonconstant, then $\Delta^+g(\tau)=0$ for every $\tau\in[a,b)$, so $\kappa_g=0$. Thus the universal $g$-H\"older selection property in Theorem~\ref{thm:clock-characterization} fails for every nonconstant continuous clock. Theorem~\ref{thm:clock-characterization} therefore recovers the classical obstruction for ordinary interval-valued time.
\end{remark}

\section{Examples and sharp distinctions}\label{sec:examples}

Throughout this section, let $I=[a,b]$ with $a<b$ and fix $0<\alpha<1$; the following examples illustrate how the jump structure of the clock determines whether the condition $\kappa_g>0$ holds.

The condition $\kappa_g>0$ separates clocks according to the distribution of their increments across scales, not merely according to whether jumps are present.

\begin{example}[Finite-step clocks]\label{ex:finite-step}
Let
\[
a<\tau_1<\cdots<\tau_N<b,
\qquad m_1,\ldots,m_N>0,
\]
and define $g:I\to\mathbb R$ by
\[
g(t)=g(a)+\sum_{\tau_j<t}m_j.
\]
Then $\Delta^+g(\tau_j)=m_j$. For every $s,t\in I$ with $s<t$ and $g(s)<g(t)$,
\[
g(t)-g(s)=\sum_{s\le\tau_j<t}m_j,
\qquad
J_g(s,t)=\max_{s\le\tau_j<t}m_j.
\]
Hence
\[
\kappa_g
=
\min_{1\le p\le q\le N}
\frac{\max_{p\le j\le q}m_j}{\sum_{j=p}^q m_j}
\ge\frac1N.
\]
Theorem~\ref{thm:clock-characterization} therefore gives
\[
[f]_{\alpha,g}\le8N^\alpha[F]_{\alpha,g}.
\]
If $m_1=\cdots=m_N$, then $\kappa_g=1/N$, so the lower bound for this class is attained.
\end{example}

\begin{example}[Geometrically decaying jumps]\label{ex:geometric}
Let $\tau_n\uparrow b$, fix $q\in(0,1)$, and set $m_n=(1-q)q^{n-1}$. Define $g:I\to\mathbb R$ by
\[
g(t)=\sum_{\tau_n<t}m_n.
\]
Then $g$ is a bounded, left-continuous, nondecreasing pure-jump clock. For integers $1\le n\le m$, take $s=\tau_n$ and $t=\tau_{m+1}$. Then
\[
\frac{J_g(s,t)}{g(t)-g(s)}
=
\frac{1-q}{1-q^{m-n+1}}
\ge1-q.
\]
For $s=\tau_n$ and $t=b$, the interval $[s,t)$ contains the whole tail beginning with $m_n$, and the ratio equals $1-q$. For arbitrary $s<t$, the jumps contained in $[s,t)$ form either a finite consecutive block $\tau_n,\ldots,\tau_m$ or a tail of the sequence; hence the two cases above exhaust all possibilities relevant to the infimum defining $\kappa_g$. Therefore
\[
\boxed{\kappa_g=1-q.}
\]
Consequently,
\[
[f]_{\alpha,g}
\le
8(1-q)^{-\alpha}[F]_{\alpha,g}.
\]
\end{example}

\begin{example}[Polynomially decaying jumps]\label{ex:polynomial}
Let $\tau_n\uparrow b$ and, for $r>1$, set $m_n=n^{-r}$. Define $g:I\to\mathbb R$ by
\[
g(t)=\sum_{\tau_n<t}m_n.
\]
The series converges, so $g$ is again a bounded pure-jump clock. For $n\ge1$, take $s=\tau_n$ and $t=\tau_{2n+1}$, so that the jumps in $[s,t)$ are exactly $m_n,m_{n+1},\ldots,m_{2n}$. Then
\[
J_g(s,t)=n^{-r},
\]
while
\[
\sum_{k=n}^{2n}k^{-r}
\ge
n(2n)^{-r}
=
2^{-r}n^{1-r}.
\]
Hence
\[
\frac{J_g(s,t)}{g(t)-g(s)}
\le
\frac{2^r}{n}\to0,
\]
and therefore
\[
\boxed{\kappa_g=0.}
\]
Thus pure-jump structure alone does not imply the universal H\"older-selection property.
\end{example}

\begin{example}[The ordinary continuous clock]\label{ex:continuous}
Let $I=[0,1]$ and define $g:I\to\mathbb R$ by $g(t)=t$. Then $J_g(s,t)=0$ for every $s<t$, whereas $g(t)-g(s)=t-s>0$. Thus
\[
\kappa_g=0.
\]
Theorem~\ref{thm:clock-characterization} recovers the failure of a universal $\alpha$-H\"older selection principle for $0<\alpha<1$ on an ordinary interval.
\end{example}

\begin{example}[Adding one jump to continuous growth]\label{ex:mixed}
Let $I=[0,1]$, choose $0<\tau<1$ and $\lambda>0$, and define $g:I\to\mathbb R$ by
\[
g(t)=t+\lambda\mathbf 1_{(\tau,1]}(t).
\]
Here, $\mathbf 1_{(\tau,1]}$ denotes the indicator function of the interval $(\tau,1]$. Then $g$ is left-continuous and nondecreasing, with $\Delta^+g(\tau)=\lambda$. However, for any $s,t\in I$ with $s<t$ and $\tau\notin[s,t)$,
\[
J_g(s,t)=0,
\qquad
g(t)-g(s)=t-s>0.
\]
Hence $\kappa_g=0$. A large isolated jump cannot compensate for positive continuous growth occurring on other scales.
\end{example}

Examples~\ref{ex:geometric} and \ref{ex:polynomial} give a direct qualitative contrast. Both are bounded, left-continuous, nondecreasing pure-jump clocks, yet only the clock whose jump sizes decay geometrically satisfies the scale-independent dominance condition required in Theorem~\ref{thm:clock-characterization}.

\section{Structural consequences of uniform jump dominance}\label{sec:structural}

In this section, we examine the measure-theoretic consequences of the condition $\kappa_g>0$. We show that uniform jump dominance forces the Lebesgue--Stieltjes measure associated with $g$ to be purely atomic, and we derive the corresponding obstruction when a nonatomic component is present. We also clarify that pure atomicity alone is not sufficient for $\kappa_g>0$.

\begin{proposition}[Uniform jump dominance forces a purely atomic clock]\label{prop:pure-atomic}
Let $g:I\to\mathbb R$ be nonconstant, left-continuous, and nondecreasing. If $\kappa_g>0$, then
\[
\mu_g^c=0.
\]
Equivalently, for every $a\le s<t\le b$,
\[
g(t)-g(s)
=
\sum_{\tau\in D_g\cap[s,t)}\Delta^+g(\tau).
\]
\end{proposition}

\begin{proof}
Choose a constant $\kappa'$ satisfying $0<\kappa'<\kappa_g$. We use only intervals of the forms $[p,q)$ and $(p,q)$ with $a\le p<q\le b$. Let $J$ be such an interval and suppose that $\mu_g(J)>0$.

If $J=[p,q)$, take $J'=J$. If $J=(p,q)$, choose $p_n\downarrow p$ with $p_n\in(p,q)$ and use continuity from below to obtain
\[
\mu_g([p_n,q))\uparrow\mu_g((p,q)).
\]
Thus in either case there is a half-open interval $J'=[r,q)\subset J$ such that
\[
\mu_g(J')>
\frac{\kappa'}{\kappa_g}\,\mu_g(J).
\]
By the definition of $\kappa_g$,
\[
\sup_{\tau\in J'}\mu_g(\{\tau\})
=
\sup_{\tau\in J'}\Delta^+g(\tau)
\ge
\kappa_g\mu_g(J')
>
\kappa'\mu_g(J).
\]
Since the supremum is strictly larger than $\kappa'\mu_g(J)$, there exists an atom $\tau_J\in J'\subset J$ such that
\[
\mu_g(\{\tau_J\})>
\kappa'\mu_g(J).
\]

Fix $s,t\in I$ with $s<t$ and $\mu_g([s,t))>0$, and set $J_0=[s,t)$. Select $\tau_{J_0}$ as above and remove it. The remainder is the disjoint union of at most two intervals of the forms $[p,q)$ or $(p,q)$, and its total measure is less than
\[
(1-\kappa')\mu_g(J_0).
\]
Apply the same procedure to every positive-measure interval in this remainder. After $n$ generations there are at most $2^n$ residual intervals, and the sum of their measures is at most
\[
(1-\kappa')^n\mu_g(J_0).
\]
Let $R_n$ be the union of the residual intervals after the $n$th generation. Then $R_{n+1}\subset R_n$ and
\[
\mu_g(R_n)
\le
(1-\kappa')^n\mu_g(J_0)
\longrightarrow0.
\]
The selected atoms over all generations form a countable set $A\subset D_g\cap[s,t)$, and
\[
J_0\setminus A=\bigcap_{n\ge0}R_n.
\]
Since $\mu_g(J_0)<\infty$, continuity from above gives
\[
\mu_g(J_0\setminus A)
=
\lim_{n\to\infty}\mu_g(R_n)=0.
\]
Therefore
\[
\mu_g([s,t))
=
\sum_{\tau\in A}\mu_g(\{\tau\})
=
\sum_{\tau\in D_g\cap[s,t)}\Delta^+g(\tau).
\]
Indeed, any unselected point of $D_g\cap[s,t)$ would be an atom of positive $\mu_g$-mass contained in the null set $J_0\setminus A$. Since $s<t$ were arbitrary, the nonatomic part of $\mu_g$ vanishes, and hence $\mu_g^c=0$.
\end{proof}

\begin{corollary}\label{cor:continuous-mass-obstruction}
Let $0<\alpha<1$. If $\mu_g^c\ne0$, then $\kappa_g=0$. Consequently, there exist a metric space $(X,d)$ and a compact-valued $g$-H\"older map $F:I\to\K(X)$ of exponent $\alpha$ that admits no $g$-H\"older selection.
\end{corollary}

\begin{proof}
The first assertion is the contrapositive of Proposition~\ref{prop:pure-atomic}; the second follows from Theorem~\ref{thm:clock-characterization}.
\end{proof}

\begin{remark}\label{rem:pure-jump-not-sufficient}
The converse of Proposition~\ref{prop:pure-atomic} is false. Example~\ref{ex:polynomial} is purely atomic but has $\kappa_g=0$. Thus the decisive property is quantitative atomic separation across scales, not atomicity alone.
\end{remark}

\section{Discussion and conclusions}\label{sec:discussion}

The question guiding this study was which Stieltjes clocks allow H\"older regularity to pass from a compact-valued multifunction to one of its selections when $0<\alpha<1$. The results give a precise answer: this happens exactly when the clock image $\ImG$ is uniformly disconnected. Thus, the existence of a H\"older selection is determined not only by the regularity of the multifunction, but also by the way the clock values are distributed.

For a left-continuous and nondecreasing Stieltjes clock, the same condition can be read directly from its jumps. The requirement $\kappa_g>0$ means that every positive increase of the clock contains a jump representing a fixed positive fraction of that increase. Under this condition, a $g$-H\"older selection can be chosen through any prescribed point of the graph, with explicit control of its H\"older regularity.

The examples show why the presence of jumps alone is not sufficient. Finite-step clocks satisfy the required condition, and clocks with infinitely many jumps may also satisfy it. In particular, geometrically decaying jumps can preserve the required dominance across scales. By contrast, for some clocks with polynomially decaying jumps, no single jump remains large enough relative to the accumulated increase. These examples show that the distribution of the jumps is more important than their mere existence.

The same distinction appears at the level of the associated Stieltjes measure. The condition $\kappa_g>0$ forces the measure to be purely atomic, but a purely atomic measure need not satisfy $\kappa_g>0$. Hence, atomicity gives only part of the picture; the relative sizes and distribution of the atoms are also essential.

The converse result completes the characterization for $0<\alpha<1$. When the clock image does not have the required separation property, one can construct a compact-valued $g$-H\"older multifunction for which no $g$-H\"older selection exists. In this sense, the condition on the clock is both sufficient and necessary for the universal selection property considered here.

Overall, the results show that the selection problem for $0<\alpha<1$ can be decided by a concrete property of the Stieltjes clock. Uniform disconnectedness of the clock image and uniform dominance of its jumps provide two equivalent ways of identifying the clocks for which H\"older regularity can be preserved by a selection.

\section*{Statements and Declarations}

\noindent\textbf{Funding.} No funding was received for conducting this study.

\medskip
\noindent\textbf{Competing interests.} The authors have no relevant financial or non-financial interests to disclose.

\medskip
\noindent\textbf{Data availability.} No datasets were generated or analysed during the current study.

\end{document}